\RequirePackage{etoolbox}
\csdef{input@path}{{style/}{graphics/}}
\makeatletter
\input{arxiv-vmsta.cfg}
\makeatother
\documentclass[numbers,compress]{vmsta}
\usepackage{dcolumn}

\volume{1}
\pubyear{2014}
\firstpage{139}
\lastpage{149}
\doi{10.15559/15-VMSTA17}

\startlocaldefs

\allowdisplaybreaks

\theoremstyle{plain}
\newtheorem{thm}{Theorem}

\newtheorem{crt}{Criterion}

\theoremstyle{definition}
\newtheorem{remark}{Remark}
\newtheorem{defin}{Definition}
\newtheorem{ex}{Example}

\newcolumntype{d}[1]{D{.}{.}{#1}}
\DeclareMathOperator{\var}{Var}
\endlocaldefs

\begin{document}
\begin{frontmatter}

\title{A criterion for testing hypotheses about the covariance function
of a stationary Gaussian stochastic process}

\author[a]{\inits{Yu.}\fnm{Yuriy}\snm{Kozachenko}}\email{ykoz@ukr.net}
\author[b]{\inits{V.}\fnm{Viktor}\snm{Troshki}\corref{cor1}}\email{btroshki@ukr.net}
\cortext[cor1]{Corresponding author.}
\address[a]{Taras Shevchenko National University of Kyiv, Kyiv, Ukraine}
\address[b]{Uzhhorod National University, Uzhhorod, Ukraine}

\markboth{Yu. Kozachenko, V. Troshki}{A criterion for testing
hypotheses about the covariance function}

\begin{abstract}
We consider a measurable stationary Gaussian stochastic process.
A~criterion for testing hypotheses about the covariance function of
such a~process using estimates for its norm in the space $L_p(\mathbb
{T}),\, p\geq1$, is constructed.
\end{abstract}


\received{21 November 2014}
%
\revised{18 January 2015}
%
\accepted{19 January 2015}
\publishedonline{29 January 2015}
\end{frontmatter}

\section{Introduction}\label{sec1}

We construct a criterion for testing the hypothesis that the
covariance function of measurable real-valued stationary Gaussian
stochastic process $X(t)$ equals~$\rho(\tau)$. We shall use the correlogram
\[
\hat{\rho}(\tau)=\frac{1}{T}\int\limits
_0^{T}X(t+\tau)X(t)dt,\quad 0\leq \tau\leq T,
\]
as an estimator of the function $\rho(\tau)$.

A lot of papers so far have been dedicated to estimation of covariance
function with given accuracy in the uniform metric, in particular, the
papers \cite{B.,B.Z.,I.,K.O.,K.St.} and
the book \cite{L.I.}.
We also note that similar estimates of Gaussian stochastic processes
were obtained in books \cite{D.V.} and \cite{B.P.}. The main properties
of the correlograms of stationary Gaussian stochastic processes were
studied by Buldygin and Kozachenko~\cite{B.K.}.\looseness=1

The definition of a square Gaussian random vector was introduced by
Kozachenko and Moklyachuk \cite{K.M.}. Applications of the theory of
square Gaussian random variables and stochastic processes in
mathematical statistics were considered in the paper \cite{K.K.} and in
the book \cite{B.K.}. In the papers \cite{F.} and \cite{K.F.},
~Kozachenko and Fedoryanich constructed a criterion for testing
hypotheses about the covariance function of a Gaussian stationary
process with given accuracy and reliability in $L_2(\mathbb{T})$.

Our goal is to estimate the covariance function $\rho(\tau)$ of a
Gaussian stochastic process with given accuracy and reliability in
$L_p(\mathbb{T}),\, p\geq1$. Also, we obtain the estimate for the norm
of square Gaussian stochastic processes in the space~$L_p(\mathbb{T})$.
We use this estimate for constructing a criterion for testing
hypotheses about the covariance function of a Gaussian stochastic process.

The article is organized as follows. In Section \ref{sec2}, we give necessary
information about the square Gaussian random variables. In Section~\ref{sec3},
we obtain an estimate for the norm of square Gaussian stochastic
processes in the space~$L_p(\mathbb{T})$. In Section \ref{sec4}, we propose a
criterion for testing a hypothesis about the covariance function of a
stationary Gaussian stochastic process based on the estimate obtained
in Section~\ref{sec3}.\looseness=1
\vspace*{3pt}

\section{Some information about the square Gaussian random variables\\
and~processes}\label{sec2}
\vspace*{3pt}

\begin{defin}[\cite{B.K.}]
Let $\mathbb{T}$ be a parametric set, and let $\varXi=\{\xi_t:t\in\mathbf
{T}\}$ be a~family of Gaussian random variables such that $\mathbf{E}\xi
_t=0$. The space $\mathit{SG}_{\varXi}(\varOmega)$ is called a space of square
Gaussian random variables if any $\zeta\in \mathit{SG}_{\varXi}(\varOmega)$ can be
represented as
\[
\zeta=\bar{\xi}^{T}A\bar{\xi}-\mathbb{E}\bar{\xi}^{T}A\bar{
\xi},
\]
where $\bar{\xi}=(\xi_1,\ldots,\xi_N)^{T}$ with $\xi_k\in\varXi,\,
k=1,\ldots,n$, and $A$ is an arbitrary matrix with real-valued entries,
or if $\zeta\in \mathit{SG}_{\varXi}(\varOmega)$ has the representation
\[
\zeta=\lim\limits
_{n\rightarrow\infty} \bigl(\bar{\xi}^{T}_nA\bar{\xi
}_n-\mathbb{E}\bar{\xi}^{T}_nA\bar{
\xi}_n \bigr).
\]
\end{defin}

\begin{thm}[\cite{B.K.}]
Assume that $\zeta\in \mathit{SG}_{\varXi}(\varOmega)$ and $\var\zeta>0$. Then the
following inequality holds for $|s|<1$:
%
\begin{equation}
\label{2} \mathbf{E}\exp \biggl\{\frac{s}{\sqrt{2}} \biggl(\frac{\zeta}{\sqrt{\var
\zeta}}
\biggr) \biggr\}\leq\frac{1}{
\sqrt{1-|s|}}\exp \biggl\{-\frac{|s|}{2} \biggr
\}=L_0(s).
\end{equation}
\end{thm}

\begin{defin}[\cite{B.K.}]
A stochastic process $Y$ is called a square Gaussian stochastic process
if for each $t\in\mathbb{T}$, the random variable $Y(t)$ belongs to
the space $\mathit{SG}_{\varXi}(\varOmega)$.
\end{defin}

\section{An estimate for the $L_p(\mathbb{T})$ norm of a square
Gaussian stochastic process}\label{sec3}

In the following theorem, we obtain an estimate for the norm of square
Gaussian stochastic processes in the space $L_p(\mathbb{T})$. We shall
use this result for constructing a criterion for testing hypotheses
about the covariance function of a~Gaussian stochastic process.

\begin{thm}\label{4}
Let $\{\mathbb{T}, \mathfrak{A}, \mu\}$be a measurable space, where
$\mathbb{T}$ is a parametric set, and let $Y=\{Y(t), t\in\mathbb{T}\}$
be a square Gaussian stochastic process. Suppose that $Y$ is a
measurable process. Further, let the Lebesgue integral $\int
_{\mathbb{T}}(\mathbf{E}Y^2(t))^{\frac{p}{2}}d\mu(t)$ be well defined
for $p\geq1$. Then the integral $\int
_{\mathbb{T}}(Y(t))^pd\mu
(t)$ exists with probability~$1$, and
%
\begin{equation}
\label{1} P \biggl\{\int\limits
_{\mathbb{T}}\bigl| Y(t)\bigr|^pd\mu(t)>\varepsilon \biggr\}
\leq2\sqrt{1+\frac{\varepsilon^{1/p}\sqrt{2}}{C_p^{\frac{1}{p}}}}\exp \biggl\{-\frac{\varepsilon^{\frac{1}{p}}}{\sqrt{2}C_p^{\frac
{1}{p}}} \biggr\}
\end{equation}
for all $\varepsilon\geq (\frac{p}{\sqrt{2}}+\sqrt{(\frac
{p}{2}+1)p} )^{p}C_p$, where $C_p=\int
_{\mathbb{T}}(\mathbf
{E}Y^2(t))^{\frac{p}{2}}d\mu(t)$.
\end{thm}
\begin{proof}
Since $\max
_{x>0}x^{\alpha}e^{-x}=\alpha^{\alpha}e^{-\alpha}$,
we have $x^{\alpha}e^{-x}\leq\alpha^{\alpha}e^{-\alpha}$.

If $\zeta$ is a random variable from the space $\mathit{SG}_{\varXi}(\varOmega)$ and
$x=\frac{s}{\sqrt{2}}\cdot\frac{|\zeta|}{\sqrt{\mathbf{E}\zeta^2}}$,
where $s>0$, then
\[
\mathbf{E} \biggl(\frac{s}{\sqrt{2}}\frac{|\zeta|}{\sqrt{\mathbf{E}\zeta
^2}} \biggr)^{\alpha}\leq
\alpha^{\alpha}e^{-\alpha}\cdot\mathbf{E}\exp \biggl\{\frac{s}{\sqrt{2}}
\frac{|\zeta|}{\sqrt{\mathbf{E}\zeta^2}} \biggr\}
\]
and
\[
\mathbf{E}|\zeta|^{\alpha}\leq \biggl(\frac{\sqrt{2\mathbf{E}\zeta
^2}}{s}
\biggr)^{\alpha}\alpha^{\alpha}e^{-\alpha}\mathbf{E}\exp \biggl\{
\frac{s}{\sqrt{2}}\frac{|\zeta|}{\sqrt{\mathbf{E}\zeta^2}} \biggr\}.
\]
From inequality (\ref{2}) for $0<s<1$ we get that
\begin{align}
\label{6} \mathbf{E}|\zeta|^{\alpha}&\leq \biggl(\frac{\sqrt{2\mathbf{E}\zeta
^2}}{s}
\biggr)^{\alpha}\alpha^{\alpha}e^{-\alpha} \biggl(\mathbf{E}\exp
\biggl\{\frac{s}{\sqrt{2}}\frac{\zeta}{\sqrt{\mathbf{E}\zeta^2}} \biggr\}+ \mathbf{E}\exp \biggl\{-
\frac{s}{\sqrt{2}}\frac{\zeta}{\sqrt{\mathbf
{E}\zeta^2}} \biggr\} \biggr)
\nonumber
\\
& \leq\frac{2}{\sqrt{1-s}} \biggl(\frac{\sqrt{2\mathbf{E}\zeta^2}}{s} \biggr)^{\alpha}
\alpha^{\alpha}e^{-\alpha}\exp \biggl\{-\frac{s}{\sqrt{2}} \biggr\}
\nonumber
\\
& =2L_0(s) \biggl(\frac{\sqrt{2\mathbf{E}\zeta^2}}{s} \biggr)^{\alpha}\alpha
^{\alpha}e^{-\alpha}.
\end{align}

Let $Y(t),\, t\in\mathbb{T}$, be a measurable square Gaussian
stochastic process. Using the Chebyshev inequality, we derive that, for
all $l\geq1$,
\[
P \biggl\{\int\limits
_{\mathbb{T}}\bigl| Y(t)\bigr|^pd\mu(t)>\varepsilon \biggr\}\leq
\frac{\mathbf{E} (\int
_{\mathbb{T}}| Y(t)|^pd\mu(t)
)^l}{\varepsilon^l}.
\]
Then from the generalized Minkowski inequality together with
inequality~(\ref{6}) for \mbox{$l>1$} we obtain that
\begin{align}
& \biggl(\mathbf{E} \biggl(\int\limits
_{\mathbb{T}}\bigl| Y(t)\bigr|^pd\mu(t)
\biggr)^l \biggr)^{\frac{1}{l}}\leq\int\limits
_{\mathbb{T}} \bigl(\mathbf{E}\bigl|
Y(t)\bigr|^{pl} \bigr)^{\frac{1}{l}}d\mu(t)
\nonumber
\\
& \leq\int\limits
_{\mathbb{T}}\bigl(2L_0(s) \bigl(2\mathbf{E}Y^2(t)
\bigr)^{\frac
{pl}{2}}(pl)^{pl}s^{-pl}\exp\{-pl\}
\bigr)^{\frac{1}{l}}d\mu(t)
\nonumber
\\
& =\bigl(2L_0(s)\bigr)^{\frac{1}{l}}\int\limits
_{\mathbb{T}}\bigl(2\mathbf
{E}Y^2(t)\bigr)^{\frac{p}{2}}s^{-p}(pl)^p\exp
\{-p\}d\mu(t)
\nonumber
\\
& =\bigl(2L_0(s)\bigr)^{\frac{1}{l}}2^{\frac{p}{2}}s^{-p}(pl)^p
\exp\{-p\}\int\limits
_{\mathbb{T}}\bigl(\mathbf{E}Y^2(t)\bigr)^{\frac{p}{2}}d
\mu(t).
\nonumber
\end{align}

Assuming that $C_p=\int
_{\mathbb{T}}(\mathbf{E}Y^2(t))^{\frac
{p}{2}}d\mu(t)$, we deduce that
\[
\mathbf{E} \biggl(\int\limits
_{\mathbb{T}}\bigl| Y(t)\bigr|^pd\mu(t)
\biggr)^l\leq 2L_0(s)2^{\frac{pl}{2}}(lp)^{pl}
\exp\{-pl\}C_p^ls^{-pl}.
\]
Hence,
\begin{align}
P \biggl\{\int\limits
_{\mathbb{T}}\bigl| Y(t)\bigr|^pd\mu(t)>\varepsilon \biggr\} &\leq2
\cdot\bigl(2^{\frac{p}{2}}\bigr)^lL_0(s)
\bigl(p^p\bigr)^l\bigl(\exp\{-p\} \bigr)^lC_p^l
\bigl(s^{-p}\bigr)^l\cdot\frac{(l^p)^l}{\varepsilon^l}
\nonumber
\\
& =2L_0(s) a^l\bigl(l^p\bigr)^l
\nonumber
,
\end{align}
where $a=\frac{2^{\frac{p}{2}}p^pC_p}{e^ps^p\varepsilon}$, that is,
$a^{\frac{1}{p}}=\frac{2^{\frac{1}{2}}pC_p^{\frac{1}{p}}}{es\varepsilon
^{\frac{1}{p}}}$. Let us find the minimum of the function $\psi
(l)=a^l(l^p)^l$. We can easily check that it reaches its minimum at the point
$l^{*}=\frac{1}{ea^{\frac{1}{p}}}$.

Then
\begin{align}
2L_0(s) \psi\bigl(l^{*}\bigr)&=2L_0(s)
a^{\frac{1}{ea^{\frac{1}{p}}}}\cdot \biggl(\frac{1}{ea^{\frac{1}{p}}} \biggr)^{p\cdot\frac{1}{ea^{\frac
{1}{p}}}}=2L_0(s)
a^{\frac{1}{ea^{\frac{1}{p}}}}\cdot a^{-\frac
{1}{ea^{\frac{1}{p}}}}\cdot e^{-\frac{p}{ea^{\frac{1}{p}}}}
\nonumber
\\
& =2L_0(s) \exp \biggl\{-\frac{pes\varepsilon^{\frac{1}{p}}}{2^{\frac
{1}{2}}peC_p^{\frac{1}{p}}} \biggr
\}=2L_0(s) \exp \biggl\{-\frac
{s\varepsilon^{\frac{1}{p}}}{2^{\frac{1}{2}}C_p^{\frac{1}{p}}} \biggr\}
\nonumber
\\
& =\frac{2}{\sqrt{1-s}}\exp \biggl\{-s \biggl(\frac{1}{2}+\frac{\varepsilon
^{1/p}}{2^{\frac{1}{2}}C_p^{\frac{1}{p}}}
\biggr) \biggr\}.
\nonumber
\end{align}

In turn, minimizing the function $\theta(s)=\frac{2}{\sqrt{1-s}}\exp
 \{-s (\frac{1}{2}+\frac{\varepsilon^{1/p}}{2^{\frac
{1}{2}}C_p^{\frac{1}{p}}} ) \}$ in~$s$, we deduce
$s^*=1-\frac{1}{1+\frac{\sqrt{2}\varepsilon^{1/p}}{C_p^{1/p}}}$. Thus,
\[
\theta\bigl(s^*\bigr)=2\sqrt{1+\frac{\varepsilon^{1/p}\sqrt{2}}{C_p^{\frac
{1}{p}}}}\exp \biggl\{-
\frac{\varepsilon^{\frac{1}{p}}}{\sqrt{2}C_p^{\frac
{1}{p}}} \biggr\}.
\]\eject\noindent
Since $l^{*}\geq1$, it follows that inequality (\ref{1}) holds if $\frac
{1}{ea^{\frac{1}{p}}}=\frac{s\varepsilon^{1/p}}{\sqrt{2}pC_p^{1/p}}\geq
1$. Substituting the value of $s^*$ into this expression, we obtain the
inequality $\varepsilon^{2/p}\geq pC_p^{1/p}(C_p^{1/p}+\sqrt
{2}\varepsilon^{1/p})$. Solving this inequality with respect to
$\varepsilon>0$, we deduce that inequality (\ref{1}) holds for
$\varepsilon\geq (\frac{p}{\sqrt{2}}+\sqrt{(\frac{p}{2}+1)p}
)^{p}C_p$.
The theorem is proved.\vspace*{3pt}
\end{proof}

\section{The construction of a criterion for testing hypotheses about
the covariance\\ function of a stationary Gaussian stochastic
process}\label{sec4}\vspace*{3pt}

Consider a measurable stationary Gaussian stochastic process $X$
defined for all $t\in\mathbb{R}$. Without any loss of generality, we
can assume that $X=\{X(t)$, $t\in\mathbb{T}=[0,T+A]$, $0<T<\infty$,
$0<A<\infty\}$ and $\mathbf{E}X(t)=0$. The covariance function $\rho(\tau
)=\mathbf{E}X(t+\tau)X(t)$ of this process is defined for any $\tau\in
\mathbb{R}$ and is an even function. Let $\rho(\tau)$ be continuous on
$\mathbb{T}$.\vspace*{3pt}
\begin{thm}\label{e}
Let the correlogram
%
\begin{equation}
\label{5} \hat{\rho}(\tau)=\frac{1}{T}\int\limits
_0^{T}X(t+
\tau)X(t)dt,\quad 0\leq \tau\leq A,
\end{equation}
be an estimator of the covariance function $\rho(\tau)$. Then the
following inequality holds for all $\varepsilon\geq (\frac{p}{\sqrt
{2}}+\sqrt{(\frac{p}{2}+1)p} )^{p}C_p$\emph{:}
\[
P \biggl\{\int\limits
_0^{A}\bigl(\hat{\rho}(\tau)-\rho(\tau)
\bigr)^pd\tau >\varepsilon \biggr\}\leq2\sqrt{1+
\frac{\varepsilon^{1/p}\sqrt
{2}}{C_p^{\frac{1}{p}}}}\exp \biggl\{-\frac{\varepsilon^{\frac
{1}{p}}}{\sqrt{2}C_p^{\frac{1}{p}}} \biggr\},
\]
where $C_p=\int
_0^{A} (\frac{2}{T^2}\int
_0^{T}(T-u)(\rho^2(u)+\rho(u+\tau)\rho(u-\tau))du )^{\frac
{p}{2}}d\tau$ and $0<A<~\infty$.\vspace*{3pt}
\end{thm}

\begin{remark}
Since the sample paths of the process $X(t)$ are continuous with
probability one on the set $\mathbb{T}$, $\hat{\rho}(\tau)$ is a
Riemann integral.\vspace*{3pt}
\end{remark}

\begin{proof}
Consider
\[
\mathbf{E}\bigl(\hat{\rho}(\tau)-\rho(\tau)\bigr)^2=\mathbf{E}\bigl(
\hat{\rho}(\tau )\bigr)^2-\rho^2(\tau).
\]
From the Isserlis equality for jointly Gaussian random variables it
follows that
\begin{align}
&{}\mathbf{E}\bigl(\hat{\rho}(\tau)\bigr)^2-\rho^2(\tau)=
\mathbf{E} \biggl(\frac
{1}{T^2}\int\limits
_0^{T}\int\limits
_0^{T}X(t+
\tau)X(t)X(s+\tau )X(s)dtds \biggr)-\rho^2(\tau)
\nonumber
\\
&\quad{} =\frac{1}{T^2}\int\limits
_0^{T}\int\limits
_0^{T}\bigl(
\mathbf{E}X(t+\tau )X(t)\mathbf{E}X(s+\tau)X(s)+\mathbf{E}X(t+\tau)X(s+\tau)
\nonumber
\\
&\qquad{} \times\mathbf{E}X(t)X(s)+\mathbf{E}X(t+\tau)X(s)\mathbf{E}X(s+\tau )X(t)
\bigr)dtds-\rho^2(\tau)
\nonumber
\\
&\quad{} =\frac{1}{T^2}\int\limits
_0^{T}\int\limits
_0^{T}\bigl(
\rho^2(\tau)+\rho ^2(t-s)+\rho(t-s+\tau)\rho(t-s-\tau)
\bigr)dtds-\rho^2(\tau)
\nonumber
\\
&\quad{} =\frac{1}{T^2}\int\limits
_0^{T}\int\limits
_0^{T}\bigl(
\rho^2(t-s)+\rho (t-s+\tau)\rho(t-s-\tau)\bigr)dtds
\nonumber
\\
&\quad{} =\frac{2}{T^2}\int\limits
_0^{T}(T-u) \bigl(\rho^2(u)+
\rho(u+\tau)\rho(u-\tau )\bigr)du.
\nonumber
\end{align}
We have obtained that
%
\begin{equation}
\label{7} \mathbf{E}\bigl(\hat{\rho}(\tau)-\rho(\tau)\bigr)^2=
\frac{2}{T^2}\int\limits
_0^{T}(T-u) \bigl(\rho^2(u)+
\rho(u+\tau)\rho(u-\tau)\bigr)du.
\end{equation}
Since $\hat{\rho}(\tau)-\rho(\tau)$ is a square Gaussian stochastic
process (see Lemma~3.1, Chapter~6 in \cite{B.K.}), it follows from
Theorem \ref{4} that
\[
P \biggl\{\int\limits
_0^{A}\bigl(\hat{\rho}(\tau)-\rho(\tau)
\bigr)^pd\tau >\varepsilon \biggr\}\leq2\sqrt{1+
\frac{\varepsilon^{1/p}\sqrt
{2}}{C_p^{\frac{1}{p}}}}\exp \biggl\{-\frac{\varepsilon^{\frac
{1}{p}}}{\sqrt{2}C_p^{\frac{1}{p}}} \biggr\}.
\]
Applying Eq.~(\ref{7}), we get
\[
C_p= \int\limits
_0^{A} \biggl(\frac{2}{T^2}\int
\limits_0^{T}(T-u)
\bigl(\rho ^2(u)+\rho(u+\tau)\rho(u-\tau)\bigr)du
\biggr)^{\frac{p}{2}}d\tau.
\]
The theorem is proved.
\end{proof}

Denote
\[
g(\varepsilon)=2\sqrt{1+\frac{\varepsilon^{1/p}\sqrt{2}}{C_p^{\frac
{1}{p}}}}\exp \biggl\{-
\frac{\varepsilon^{\frac{1}{p}}}{\sqrt{2}C_p^{\frac
{1}{p}}} \biggr\}.
\]
From Theorem \ref{e} it follows that if $\varepsilon\geq z_p=C_p
(\frac{p}{\sqrt{2}}+\sqrt{(\frac{p}{2}+1)p} )^{p}$, then
\[
P \biggl\{\int\limits
_0^{A}\bigl(\hat{\rho}(\tau)-\rho(\tau)
\bigr)^pd\tau >\varepsilon \biggr\}\leq g(\varepsilon).
\]
Let $\varepsilon_{\delta}$ be a solution of the equation $g(\varepsilon
)=\delta$, $0<\delta<1$. Put $S_{\delta}=\max\{\varepsilon_\delta, z_p\}
$. It is obvious that $g(S_{\delta})\leq\delta$ and
\begin{align}
\label{8} P \biggl\{\int\limits
_0^{A}\bigl(\hat{\rho}(\tau)-\rho(\tau)
\bigr)^pd\tau>S_{\delta
} \biggr\}\leq\delta.
\end{align}

Let $\mathbb{H}$ be the hypothesis that the covariance function of a
measurable real-valued stationary Gaussian stochastic process $X(t)$
equals $\rho(\tau)$ for $0\leq\tau\leq A$. From Theorem \ref{e} and
(\ref{8}) it follows that to test the hypothesis $\mathbb{H}$, we can
use the following criterion.

\begin{crt}\label{a}
For a given level of confidence $\delta$ the hypothesis $\mathbb{H}$ is
accepted if
\[
\int\limits
_0^{A}\bigl(\hat{\rho}(\tau)-\rho(\tau)
\bigr)^pd\mu(\tau)<S_{\delta};
\]
otherwise, the hypothesis is rejected.
\end{crt}

\begin{remark}
The equation $g(\varepsilon)=\delta$ has a solution for any $\delta>0$
since $g(\varepsilon)$ is a~decreasing function. We can find the
solution of the equation using numerical methods.
\end{remark}

\begin{remark}
We can easily see that Criterion \ref{a} can be used if $C_p\rightarrow
0$ as $T\rightarrow\infty$.
\end{remark}

The next theorem contain assumptions under which $C_p\to0$ as
$T\rightarrow\infty$.
\begin{thm}
Let $\rho(\tau)$ be the covariance function of a centered stationary
random process. Let $\rho(\tau)$ be a continuous function. If $\rho
(T)\rightarrow0$ as $T\rightarrow\infty$, then $C_p\rightarrow0$ as
$T\rightarrow\infty$, where $C_p=\int
_0^{A}(\psi(T,\tau))^{p/2}dt$ and
\[
\psi(T,\tau)=\frac{2}{T^2}\int\limits
_0^{T}(T-u) \bigl(
\rho^2(u)+\rho(u+\tau )\rho(u-\tau)\bigr)du,\quad A>0,\ T>0.
\]
\end{thm}
\begin{proof}
We have $\psi(T,\tau)\leq\frac{2}{T}\int
_0^{T}(\rho^2(u)+\rho
(u+\tau)\rho(u-\tau))du\leq4\rho^2(0)$. Now it is suffices to prove that
$\psi(T,\tau)\rightarrow0$ as $T\rightarrow\infty$.
From the L'Hopital's rule it follows that
\begin{align}
\lim\limits
_{T\rightarrow\infty}\psi(T,\tau)&= \lim\limits
_{T\rightarrow\infty}\frac{2}{T}\int\limits
_0^{T}
\bigl(\rho^2(u)+\rho(u+\tau )\rho(u-\tau)\bigr)du
\nonumber
\\
& = \lim\limits
_{T\rightarrow\infty}\bigl(\rho^2(T)+\rho(T+\tau)\rho(T-\tau )
\bigr)=0.
\nonumber
\end{align}

Application of Lebesgue's dominated convergence theorem completes the proof.
\end{proof}

Here are examples in which we find the estimates for $C_p$.

\begin{ex}
Let $\mathbb{H}$ be the hypothesis that the covariance function of a
centered measurable stationary Gaussian stochastic process equals $\rho
(\tau)=B\exp\{-a|\tau|\}$, where $B>0$ and $a>0$.

To test the hypothesis $\mathbb{H}$, we can use Criterion \ref{a} by
selecting $\hat{\rho}_T(\tau)$ that is defined in (\ref{5}) as an
estimator of the function $\rho(\tau)$. Let $0<A<\infty$. We shall find
the value of the expression
\begin{align*}
I&{}=\int\limits_0^{T}(T-u) \bigl(e^{-2au}+e^{-a| u+\tau|}e^{-a| u-\tau|}\bigr)du \\
 &{}=\int\limits_0^{T}Te^{-2au}du+T\int\limits_0^{T}e^{-a| u+\tau|}e^{-a| u-\tau|}du -\int\limits_0^{T}ue^{-2au}du
 \\
 &\quad{}
 -\int\limits_0^{T}ue^{-a| u+\tau|}e^{-a| u-\tau|}du\nonumber\\
 &{} =I_1+I_2+I_3+I_4.
\nonumber
\end{align*}
Now lets us calculate the summands:
\begin{align*}
I_1&{}=T\int\limits_0^{T}e^{-2au}du=\frac{T}{2a} \bigl(1-e^{-2aT} \bigr),\\
I_2&{}=T\int\limits_0^{T}e^{-a| u+\tau|}e^{-a| u-\tau|}du\\
   &{}=T\biggl(\int \limits_0^{\tau}e^{-a( u+\tau)}e^{a( u-\tau)}du + \int\limits_\tau^{T}e^{-a( u+\tau)}e^{-a( u-\tau)}du \biggr)\\
   &{}=T \biggl(\int\limits_0^{\tau}e^{-2a\tau}du+\int\limits_\tau^{T}e^{-2au}du \biggr)\\
   &{}=T \biggl(\tau e^{-2a\tau}-\frac{1}{2a}e^{-2aT}+\frac{1}{2a}e^{-2a\tau} \biggr),\\
I_3&{}=\int\limits_0^{T}ue^{-2au}du=-\frac{T}{2a}e^{-2aT}+\frac{1}{2a}\int\limits_0^{T}e^{-2au}du\\
   &{}=-\frac{T}{2a}e^{-2aT}-\frac{1}{4a^2}e^{-2aT}+\frac{1}{4a^2},\\
I_4&{}=\int\limits_0^{T}ue^{-a| u+\tau|}e^{-a| u-\tau|}du\\
   &{}=\int\limits_0^{\tau}ue^{-a( u+\tau)}e^{a( u-\tau)}du +\int\limits_\tau^{T}ue^{-a( u+\tau)}e^{-a( u-\tau)}du\\
   &{}=\int\limits_0^{\tau}ue^{-2a\tau}du+\int\limits_\tau^{T}ue^{-2au}du\\
   &{}=\frac{\tau^2}{2}e^{-2a\tau}-\frac{T}{2a}e^{-2aT}+\frac{\tau}{2a}e^{-2a\tau}-\frac{1}{4a^2}e^{-2aT}+\frac{1}{4a^2}e^{-2a\tau}.
\end{align*}
Therefore,
\begin{align}
I&{}= \biggl(T\tau+\frac{T}{2a}-\frac{\tau^2}{2}-\frac{\tau}{2a}-
\frac
{1}{4a^2} \biggr)e^{-2a\tau}+\frac{1}{2a^2}e^{-2aT}+
\frac{T}{2a}-\frac
{1}{4a^2}
\nonumber
\\
&{} \leq \biggl(T\tau+\frac{T}{2a} \biggr)e^{-2a\tau}+
\frac{T}{2a}+\frac
{1}{2a^2}e^{-2aT}.
\nonumber
\end{align}

Thus, we obtain
\begin{align}
C_p&{}\leq \biggl(\frac{2B}{T^2} \biggr)^\frac{p} {2}\int
\limits_0^{A}
\biggl( \biggl(T\tau+\frac{T}{2a} \biggr)e^{-2a\tau}+
\frac{T}{2a}+\frac
{1}{2a^2}e^{-2aT} \biggr)^{p/2}d\tau
\nonumber
\\
&{}= (2B )^\frac {p} {2}\frac{T^{p/2}}{T^p}I_5= (2B
)^\frac{p} {2}\frac
{1}{T^{p/2}}I_5,
\nonumber
\end{align}
where $I_5=\int
_0^{A} ( (\tau+\frac{1}{2a}
)e^{-2a\tau}+\frac{1}{2a}+\frac{1}{2a^2}e^{-2aT} )^{p/2}d\tau$.
\end{ex}

\begin{ex}
Let $\mathbb{H}$ be the hypothesis that the covariance function of
a~centered measurable stationary Gaussian stochastic process equals
$\rho(\tau)=B\exp\{-a|\tau|^2\}$, where $B>0$ and $a>0$.

Similarly as in the previous example, to test the hypothesis $\mathbb
{H}$, we can use Criterion \ref{a} by selecting $\hat{\rho}_T(\tau)$
defined in (\ref{5}) as the estimator of the function~$\rho(\tau)$. Let
$0<A<\infty$. Let us find the value of the expression
\begin{align*}
I &{}=\int\limits_0^{T}(T-u) \bigl(e^{-2au^2}+e^{-a| u+\tau|^2}e^{-a|u-\tau|^2}\bigr)du\\
  &{}=\int\limits_0^{T}Te^{-2au^2}du +T\int\limits_0^{T}e^{-a| u+\tau|^2}e^{-a| u-\tau|^2}du-\int\limits_0^{T}ue^{-2au^2}du\\
  &\quad{}-\int\limits_0^{T}ue^{-a| u+\tau|^2}e^{-a| u-\tau|^2}du\\
  &{}=I_1+I_2+I_3+I_4.
\end{align*}
Now let us calculate the summands:
\begin{align*}
I_1&{}=T\int\limits_0^{T}e^{-2au^2}du\leq T\int\limits_0^{\infty}e^{-2au^2}du=\frac{\sqrt{\pi} T}{2\sqrt{2a}},\\
I_2&{}=T\int\limits_0^{T}e^{-a| u+\tau|^2}e^{-a| u-\tau|^2}du=Te^{-2a\tau^2}\int\limits_0^{T}e^{-2au^2}du\leq\frac{\sqrt{\pi} T}{2\sqrt{2a}}e^{-2a\tau^2},\\
I_3&{}=\int\limits_0^{T}ue^{-2au^2}du=-\frac{1}{4a}\int\limits_0^{T}e^{-2au^2}d\bigl(-2au^2\bigr)=-\frac{1}{4a} \bigl(e^{-2aT^2}-1\bigr),\\
I_4&{}=\int\limits_0^{T}ue^{-2a(u^2+\tau^2)}du=e^{-2a\tau^2}\int\limits_0^{T}ue^{-2au^2}du=-\frac{1}{4a}e^{-2a\tau^2}\bigl(e^{-2aT^2}-1 \bigr).
\end{align*}
Hence,
\begin{align}
I&\leq\frac{\sqrt{\pi} T}{2\sqrt{2a}}+\frac{\sqrt{\pi} T}{2\sqrt
{2a}}e^{-2a\tau^2}+\frac{1}{4a}
\bigl(e^{-2aT^2}-1 \bigr)+\frac
{1}{4a}e^{-2a\tau^2}
\bigl(e^{-2aT^2}-1 \bigr)
\nonumber
\\
& \leq T \biggl(\frac{\sqrt{\pi}}{2\sqrt{2a}}+\frac{\sqrt{\pi}}{2\sqrt
{2a}}e^{-2a\tau^2} \biggr).
\nonumber
\end{align}
Thus, we obtain
\[
C_p\leq \biggl(\frac{2B}{T^2} \biggr)^\frac{p} {2}\int
\limits_0^{A}
\biggl(T \biggl(\frac{\sqrt{\pi}}{2\sqrt{2a}}+\frac{\sqrt{\pi}}{2\sqrt
{2a}}e^{-2a\tau^2} \biggr)
\biggr)^{p/2}d\tau= (2B )^\frac {p} {2}\frac{1}{T^{p/2}}I_6,
\]
where $I_6=\int
_0^{A} (\frac{\sqrt{\pi}}{2\sqrt{2a}}+\frac
{\sqrt{\pi}}{2\sqrt{2a}}e^{-2a\tau^2} )^{p/2}d\tau$.
\end{ex}

\section*{Acknowledgments}
The authors express their gratitude to the referee and Professor Yu.
Mishura for valuable comments that helped to improve the paper.

%

\end{document}